\documentclass{aic}

\aicAUTHORdetails{%
  title = {Induced subgraphs and tree decompositions\\
  XVIII. Obstructions to bounded pathwidth}, 
  author = {Maria Chudnovsky, Sepehr Hajebi, and Sophie Spirkl},
  plaintextauthor = {Maria Chudnovsky, Sepehr Hajebi, Sophie Spirkl},
  plaintexttitle = {Induced subgraphs and tree decompositions XVIII. Obstructions to bounded pathwidth}, 
  runningtitle = {Induced subgraphs and tree decompositions XVIII.}, 
  copyrightauthor = {M. Chudnovsky, S. Hajebi and S. Spirkl},
  keywords = {Tree decompositions, Treewidth, Path decompositions, Pathwidth, Induced subgraphs, Graph minors.},
}   

\aicEDITORdetails{%
   year={2026},
   number={8},
   received={18 September 2025},   
   published={25 September 2026},  
   doi={10.19086/aic.2026.8},      
}   

\usepackage[shortlabels]{enumitem}
\usepackage{mathdots}
\usepackage{dsfont}
\usepackage{thm-restate}

\newtheorem{theorem}{Theorem}[section]
\newtheorem{lemma}[theorem]{Lemma}

\newtheorem{observation}[theorem]{Observation}

\DeclareMathOperator{\tw}{tw}
\DeclareMathOperator{\pw}{pw}

\def\dd{\hbox{-}}   

\newcommand{\mf}{\mathfrak}
\newcommand{\mca}{\mathcal}
\newcommand{\poi}{\mathbb{N}} 

\newcounter{tbox}
\newcommand{\sta}[1]{\medskip\medskip\refstepcounter{tbox}\noindent{\parbox{\textwidth}{(\thetbox) \emph{#1}}}\vspace*{0.3cm}}
\newcommand{\mylongtitle}[1]{%
  \ifodd\value{page}%
    \protect\parbox{0.97\linewidth}{#1}\hfill%
  \else%
    \hfill\protect\parbox{0.97\linewidth}{#1}%
  \fi%
}

\begin{document}

\begin{frontmatter}[classification=text]

\title{Induced subgraphs and tree decompositions\\
XVIII. Obstructions to bounded pathwidth} 

\author[mchud]{Maria Chudnovsky\thanks{Supported by  NSF-EPSRC Grant DMS-2120644, AFOSR grant FA9550-22-1-0083 and NSF Grant DMS-2348219.}}
\author[hajebi]{Sepehr Hajebi}
\author[spirkl]{Sophie Spirkl\thanks{We acknowledge the support of the Natural Sciences and Engineering Research Council of Canada (NSERC), [funding reference number RGPIN-2020-03912].
Cette recherche a \'et\'e financ\'ee par le Conseil de recherches en sciences naturelles et en g\'enie du Canada (CRSNG), [num\'ero de r\'ef\'erence RGPIN-2020-03912]. This project was funded in part by the Government of Ontario. This research was conducted while Spirkl was an Alfred P. Sloan Fellow.}}

\begin{abstract} The \textit{pathwidth} of a graph $G$ is the smallest $w\in \poi$ such that $G$ can be constructed from a sequence of graphs, each on at most $w+1$ vertices, by gluing them together in a linear fashion. We provide a full classification of the unavoidable induced subgraphs of graphs with large pathwidth. 
 \end{abstract}
 \end{frontmatter}


\section{Introduction}

The set of all positive integers is denoted by $\poi$. Graphs in this paper have finite vertex sets, no loops and no parallel edges. For standard graph theoretic terminology, the reader is referred to \cite{diestel}. Let $G = (V(G),E(G))$ be a graph. The \textit{treewidth} of $G$, denoted $\tw(G)$, is the smallest $w\in \poi$ for which there is a tree $T$ and an assignment of a subtree $T_v$ of $T$ to each vertex $v\in V(G)$ such that:
\begin{itemize}
    \item for every edge $uv \in E(G)$, we have $V(T_u)\cap V(T_v)\neq \varnothing$; 
    \item for every vertex $x\in V(T)$, we have $|\{v\in V(G):x\in V(T_v)\}|\leq w+1$.
\end{itemize}

The \textit{pathwidth} of $G$, denoted $\pw(G)$, is defined analogously with $T$ being a path (instead of a general tree).  Two  central results  in graph minor theory are complete descriptions of unavoidable minors in graphs of large treewidth and pathwidth. These are, respectively,
planar graphs and forests: 

\begin{theorem}[Robertson and Seymour \cite{GM5}]\label{thm:RSTW}
For every planar graph  $H$, there is a constant $f_{\ref{thm:RSTW}}=f_{\ref{thm:RSTW}}(H)\in \poi$ such that every graph $G$ with $\tw(G)> f_{\ref{thm:RSTW}}$ has a minor isomorphic to $H$. Moreover, if $H$ is not planar, then no such constant exists.
\end{theorem}

\begin{theorem}[Robertson and Seymour \cite{GM1}]\label{thm:RSPW}
For every forest $H$, there is a constant $f_{\ref{thm:RSPW}}=f_{\ref{thm:RSPW}}(H)\in \poi$ such that every graph $G$ with $\pw(G)> f_{\ref{thm:RSPW}}$ has a minor isomorphic to $H$. Moreover, if $H$ is not a forest, then no such constant exists.
\end{theorem}

The goal of this series of papers is to study the same questions for induced subgraphs (instead of minors).  Here we prove an analogue of Theorem~\ref{thm:RSPW}.
Our main result, Theorem~\ref{thm:pwisg}, identifies the following as the unavoidable induced subgraphs of graphs with large pathwidth (the exact statement and all necessary definitions will be given in Section~\ref{sec:isg}):
\begin{itemize}
    \item Complete graphs and complete bipartite graphs;
    \item Subdivided binary trees and their line graphs; and
    \item ``Constellations'' that are ``interrupted'' or ``zigzagged.''
\end{itemize}

We will derive Theorem~\ref{thm:pwisg} from another result, Theorem~\ref{thm:main_tree_indm} below.  In turn, Theorem~\ref{thm:main_tree_indm} below is about ``induced minors'':  a containment relation halfway between minors and induced subgraphs. Given a graph $G$, recall that a \textit{minor} of $G$ is a graph obtained from a subgraph of $G$ by repeatedly contracting edges. When we want to define this operation on the class of simple graphs, loops and parallel edges arising in this process are removed. An \textit{induced minor} of $G$ is a graph obtained from an \textit{induced} subgraph of $G$ by repeatedly contracting edges, and removing all loops and parallel edges arising in this process.

\begin{restatable}{theorem}{maintreeindm}\label{thm:main_tree_indm}
    For all $t\in \poi$ and every forest $H$, there is a constant $f_{\ref{thm:main_tree_indm}}=f_{\ref{thm:main_tree_indm}}(t,H)$ such that every graph $G$ with $\pw(G)>f_{\ref{thm:main_tree_indm}}$ has a subgraph isomorphic to $K_{t+1}$, an induced minor isomorphic to $K_{t,t}$ or an induced minor isomorphic to $H$.
\end{restatable}

We will state our main result, Theorem~\ref{thm:pwisg}, in Section~\ref{sec:isg}. There we also show how Theorem~\ref{thm:pwisg} follows from Theorem~\ref{thm:main_tree_indm} combined with the main result of one of our earlier papers \cite{tw16} in this series. The remainder of this paper will then be devoted to the proof of Theorem~\ref{thm:main_tree_indm}.

\section{From induced minors to induced subgraphs}\label{sec:isg}

\subsection{Definitions}
The statement of our main result involves several definitions, some of which will also be used in later sections.

Let $G$ be a graph. For a graph $H$, we say that $G$ is \textit{$H$-free} if $G$ has no induced subgraph isomorphic to $H$. A \textit{stable set} in $G$ is a set of pairwise non-adjacent vertices in $G$, and a \textit{clique} in $G$ is a set of pairwise adjacent vertices in $G$. Let $X\subseteq V(G)$. We use both $X$ and $G[X]$ to denote the induced subgraph of $G$ with vertex set $X$ (also called the \textit{subgraph of $G$ induced by $X$}). We denote by $N_G(X)$ the set of all vertices in $G\setminus X$ with at least one neighbor in $X$. If $X=\{x\}$, then we write $N_G(x)$ for $N_G(\{x\})$. 
For a set $\mca{X}$ of subsets of $V(G)$, we write $V(\mca{X})=\bigcup_{X\in \mca{X}}X$. Let $X,Y\subseteq V(G)$. We say that \textit{$X$ and $Y$ are anticomplete in $G$} if $X\cap Y=\varnothing$ and there is no edge in $G$ with an end in $X$ and an end in $Y$. For $x\in V(G)$, we say that  \textit{$x$ is anticomplete to $Y$ in $G$} if $\{x\}$ and $Y$ are anticomplete in $G$.  

Let $G, H$ be graphs. An \textit{$H$-model in $G$} is a $|V(H)|$-tuple $(A_v:v\in V(H))$ of pairwise disjoint connected induced subgraphs of $G$ such that for all distinct and adjacent vertices $u,v\in V(H)$ in $H$, the sets $A_u$ and $A_v$ are not anticomplete in $G$. We call $A_v$ \textit{the branch set associated with $v$}. 
We also say that the $H$-model $(A_v:v\in V(H))$ in $G$ is \textit{induced} if for all distinct and non-adjacent vertices $u,v\in V(H)$ in $H$, the branch sets $A_u$ and $A_v$ are anticomplete in $G$. It is straightforward to observe that a graph $G$ has a minor isomorphic to a graph $H$ if and only if there is an $H$-model in $G$, and  $G$ has an induced minor isomorphic to a graph $H$ if and only if there is an induced $H$-model in $G$.

The \textit{line graph} of $G$, denoted $L(G)$, is the graph with vertex set $E(G)$ such that $e,f\in E(G)$ are adjacent in  $L(G)$ if and only if $e$ and $f$ share an end in $G$.

Let $d\in \poi$ and let $r\in \poi\cup \{0\}$. We denote by $T_{d,r}$ the unique (up to isomorphism) rooted tree of radius $r$ such that, when $r\geq 1$, the root has degree $d$ and every vertex that is neither the root nor a leaf has degree $d+1$ (see Figure~\ref{fig:regulartree}). For instance, $T_{2,r}$ is the full binary tree of radius $r$. It is well-known \cite{GM1} that for every $r\in \poi$, all subdivisions of $T_{2,2r}$ and their line graphs have pathwidth at least $r$ (see Figure~\ref{fig:binary}).

\begin{figure}[t!]
    \centering
    \includegraphics[width=0.65\linewidth]{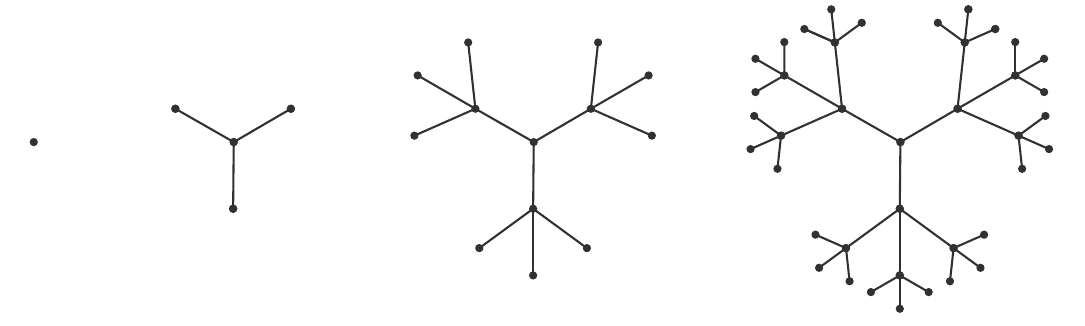}
    \caption{The trees $T_{3,r}$ for $r=0,1,2,3$.}
    \label{fig:regulartree}
\end{figure}
\begin{figure}[t!]
    \centering
    \includegraphics[width=0.7\linewidth]{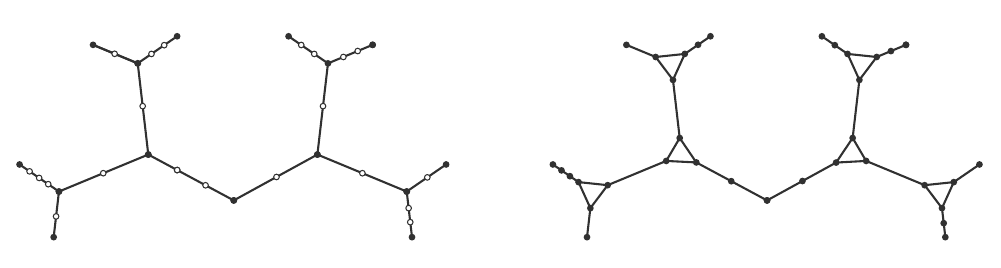}
    \caption{A subdivision of $T_{2,3}$ (left) and its line graph (right).}
    \label{fig:binary}
\end{figure}

For an integer $k$, we denote by $\poi_k$ the set of all positive integers no greater than $k$ (so $\poi_k=\varnothing$ if and only if $k\leq 0$). Let $k\in \poi$ and let $P$ be a $k$-vertex graph which is a path. Then we write $P = p_1 \dd \cdots \dd p_k$ to mean that $V(P) = \{p_1, \dots, p_k\}$ and $E(P)=\{p_ip_{i+1}:i\in \poi_{k-1}\}$. We call the vertices $p_1$ and $p_k$ the \textit{ends of $P$} and refer to $P\setminus \{p_1,p_k\}$ as the \textit{interior of $P$}, denoted $P^*$. For vertices $u,v\in V(P)$, we denote by $u\dd P\dd v$ the subpath of $P$ from $u$ to $v$. The \textit{length} of a path is the number of edges in it. It follows that a path $P$ has distinct ends if and only if $P$ has non-zero length, and $P$ has non-empty interior if and only if $P$ has length at least two. Given a graph $G$, a {\em path in $G$} is an induced subgraph of $G$ which is a path. 

A \textit{constellation} is a graph $\mf{c}$ in which there is a stable set $S_{\mf{c}}$ such that every component of $\mf{c}\setminus S_{\mf{c}}$ is a path, and each vertex $x\in S_{\mf{c}}$ has at least one neighbor in each component of $\mf{c}\setminus S_{\mf{c}}$. We denote by $\mca{L}_{\mf{c}}$ the set of all components of $\mf{c}\setminus S_{\mf{c}}$ (each of which is a path), and denote the constellation $\mf{c}$ by the pair $(S_{\mf{c}},\mca{L}_{\mf{c}})$. For $l,s\in \poi$, by an \textit{$(s,l)$-constellation} we mean a constellation $\mf{c}$ with $|S_{\mf{c}}|=s$ and $|\mca{L}_{\mf{c}}|=l$. Given a graph $G$, by an \textit{$(s,l)$-constellation in $G$} we mean an induced subgraph of $G$ which is an $(s,l)$-constellation.

We will need a few notions associated with a constellation $\mf{c} = (S_{\mf{c}}, \mathcal{L}_{\mf{c}})$, which we define below:
\begin{itemize}
    \item By a \textit{$\mf{c}$-route} we mean a path $R$ in $\mf{c}$ with ends in $S_{\mf{c}}$ and with $R^*\subseteq V(\mca{L}_{\mf{c}})$, or equivalently, with $R^*\subseteq L$ for some $L\in \mca{L}_{\mf{c}}$.
    \item For $d\in \poi$, we say that $\mf{c}$ is \textit{$d$-ample} if there is no $\mf{c}$-route of length at most $d+1$. We also say that $\mf{c}$ is \textit{ample} if $\mf{c}$ is $1$-ample. It follows that $\mf{c}$ is ample if and only if no two vertices in $S_{\mf{c}}$ have a common neighbor in $V(\mca{L}_{\mf{c}})$.
    
\item We say that $\mf{c}$ is \textit{interrupted} if there is an enumeration $x_1,\ldots, x_s$ of all vertices in $S_{\mf{c}}$ such that for all $i,j,k\in \poi_{s}$ with $i<j<k$ and every $\mf{c}$-route $R$ from $x_i$ to $x_j$, the vertex $x_k$ has a neighbor in $R$ (see Figure~\ref{fig:interrupted}).

\item For $q\in \poi$, we say that $\mf{c}$ is \textit{$q$-zigzagged} if there is an enumeration $x_1,\ldots, x_s$ of all vertices in $S_{\mf{c}}$ such that for all $i,k\in \poi_{s}$ with $i<k$ and every $\mf{c}$-route $R$ from $x_i$ to $x_k$, fewer than $q$ vertices in $\{x_j: i<j<k\}$ are anticomplete to $R$ in $\mf{c}$ (see Figure~\ref{fig:zigzag}).
\end{itemize}
\begin{figure}[t!]
    \centering
    \includegraphics[width=0.8\linewidth]{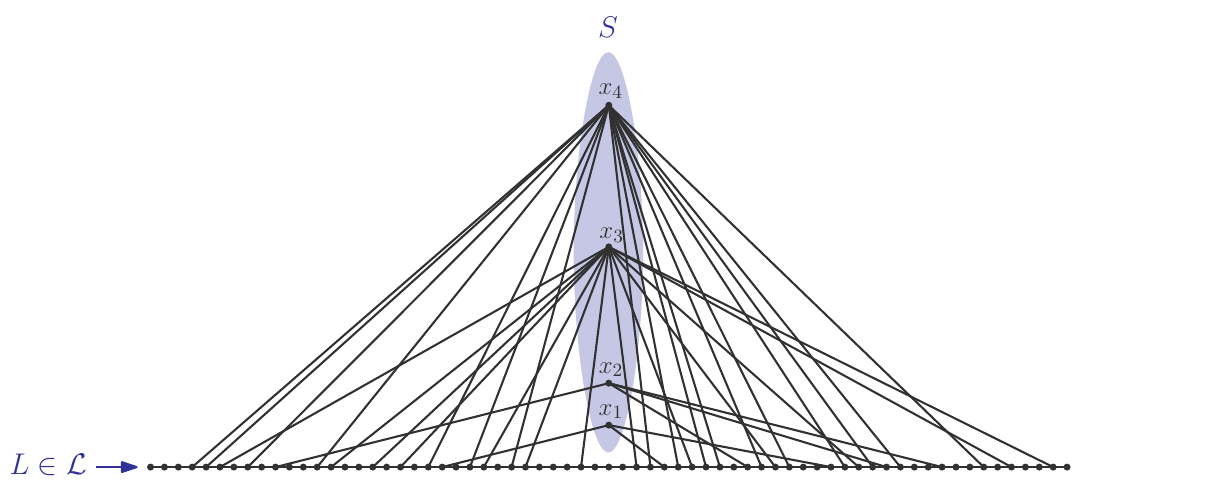}
    \caption{A $(4,1)$-constellation which is ample and interrupted.}
    \label{fig:interrupted}
\end{figure}

\begin{figure}[t!]
    \centering
    \includegraphics[width=0.8\linewidth]{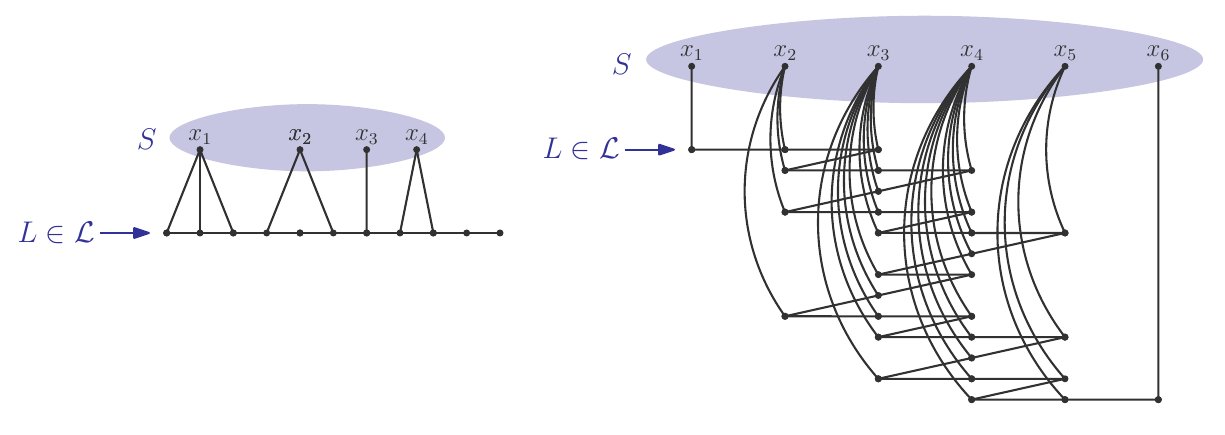}
    \caption{A $(4,1)$-constellation (left) and a $(6,1)$-constellation (right), both ample and
    $1$-zigzagged. (These constellations are \textit{not} interrupted.)}
    \label{fig:zigzag}
\end{figure}

Interrupted constellations form a slight extension of another construction from \cite{tw9, deathstar}, and zigzagged constellations are a fairly substantial generalization of a construction from \cite{davies, pohoata2014unavoidable} (see \cite{tw16} for further discussion).

\subsection{The main result}

With the above definitions in hand, we are now ready to state the main result of this paper:

\begin{theorem}\label{thm:pwisg}
    For all $d,r,l,l',s,s'\in \poi$, there is a constant $f_{\ref{thm:pwisg}}=f_{\ref{thm:pwisg}}(d,r,l,l',s,s')$ such that if $G$ is a graph with $\pw(G)>f_{\ref{thm:pwisg}}$, then one of the following holds.
       \begin{enumerate}[\rm (a)]
        \item\label{thm:pwisg_a} There is an induced subgraph of $G$ isomorphic to $K_{r+1}$, $K_{r,r}$, a subdivision of $T_{2,2r}$, or the line graph of a subdivision of $T_{2,2r}$.
        \item\label{thm:pwisg_b} There is a $d$-ample interrupted $(s,l)$-constellation in $G$. 
    \item\label{thm:pwisg_c} There is a $d$-ample $2^{4r+1}$-zigzagged $(s',l')$-constellation in $G$.
    \end{enumerate}
\end{theorem}

Theorem~\ref{thm:pwisg} is ``qualitatively'' best possible, in the sense that:
\begin{itemize}
    \item the outcomes of \ref{thm:pwisg} themselves can have arbitrarily large pathwidth; and
    \item the statement of \ref{thm:pwisg} will be false if any of the outcomes is omitted.
\end{itemize}

The first point is straightforward to check, and the second point is easily seen to be true for \ref{thm:pwisg}\ref{thm:pwisg_a}. For \ref{thm:pwisg}\ref{thm:pwisg_b} and \ref{thm:pwisg}\ref{thm:pwisg_c}, the second point follows from the two results below that we proved in \cite{tw16}, and the fact that all constellations are $K_4$-free, and all ample constellations are $K_{3,3}$-free.

\begin{theorem}[Chudnovsky, Hajebi, Spirkl \cite{tw16}]\label{thm:interruptedoutcome}
Let $\mf{c}$ be an ample interrupted constellation. Then $\mf{c}$ has no induced subgraph isomorphic to any of the following.
    \begin{itemize}
       \item An ample $q$-zigzagged $\left(3q+6,6\binom{q+2}{3}\right)$-constellation, where $q\in \poi$.
        \item A subdivision of $T_{2,7}$ or the line graph of a subdivision of $T_{2,7}$. 
    \end{itemize}
\end{theorem}

\begin{theorem}[Chudnovsky, Hajebi, Spirkl \cite{tw16}]\label{thm:zigzaggedoutcome}
Let $q\in \poi$ and let $\mf{c}$ be an ample $q$-zigzagged constellation. Then $\mf{c}$ has no induced subgraph isomorphic to any of the following.
    \begin{itemize}
        \item An ample interrupted $(2q+6,1)$-constellation.
        \item A subdivision of $T_{2,64q^2}$ or the line graph of a subdivision of $T_{2,64q^2}$.
    \end{itemize}
\end{theorem}

Additionally, one might hope that \ref{thm:pwisg}\ref{thm:pwisg_c} could be refined further to coincide with the construction of \cite{davies,pohoata2014unavoidable}, but \cite{tw16} shows that this is not possible.

As mentioned earlier, Theorem~\ref{thm:pwisg} is a consequence of Theorem~\ref{thm:main_tree_indm} combined with the main result of an earlier paper \cite{tw16} in this series. For every $r\in \poi$, we denote by $W_{r\times r}$ the $r$-by-$r$ hexagonal grid, also known as the $r$-by-$r$ \textit{wall} (see Figure~\ref{fig:wall}).

\begin{figure}
    \centering
    \includegraphics[width=0.45\linewidth]{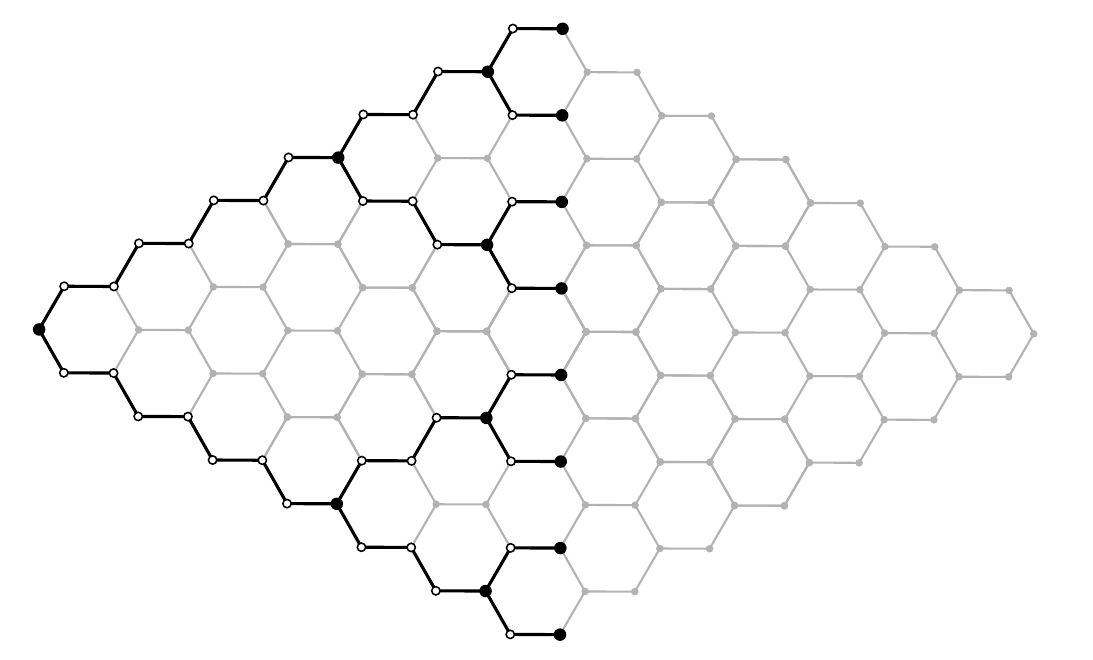}
    \caption{The graph $W_{8\times 8}$, and an induced subgraph of it isomorphic to a proper subdivision of $T_{2,3}$ (as in \ref{obs:trees}\ref{obs:trees_a} for $r=3$).}
    \label{fig:wall}
\end{figure}

\begin{theorem}[Chudnovsky, Hajebi, Spirkl \cite{tw16}]\label{thm:motherKtt}
For all $d,l,l',r,s,s'\in \poi$, there are constants $f_{\ref{thm:motherKtt}}=f_{\ref{thm:motherKtt}}(d,l,l',r,s,s')\in \poi$ and $g_{\ref{thm:motherKtt}}=g_{\ref{thm:motherKtt}}(d,l,l',r,s,s')\in \poi$ such that for every graph $G$ with an induced minor isomorphic to $K_{f_{\ref{thm:motherKtt}}, g_{\ref{thm:motherKtt}}}$, one of the following holds.
   \begin{enumerate}[\rm (a)]
        \item\label{thm:motherKtt_a} There is an induced subgraph of $G$ isomorphic to $K_{r,r}$, a subdivision of $W_{r\times r}$, or the line graph of a subdivision of $W_{r\times r}$.
        \item\label{thm:motherKtt_b} There is a $d$-ample interrupted $(s,l)$-constellation in $G$. 
    \item\label{thm:motherKtt_c} There is a $d$-ample $2r^2$-zigzagged $(s',l')$-constellation in $G$.
    \end{enumerate}
\end{theorem}

In addition to Theorem~\ref{thm:motherKtt}, we need the following observation about the presence of binary tree induced minors in walls (see Figure~\ref{fig:wall}) and of general tree induced minors in binary trees (see Figure~\ref{fig:binaryindm}; we will use \ref{obs:trees}\ref{obs:trees_b} in Section~\ref{sec:plant}). The proofs are easy, and we omit them. Recall that a \textit{proper} subdivision of a graph $H$ is a subdivision of $H$ where each edge is subdivided at least once.

\begin{observation}\label{obs:trees}
    The following hold for all $d,r\in \poi$.
    \begin{enumerate}[{\rm (a)}]
        \item\label{obs:trees_a} There is an induced subgraph of $W_{2^r\times 2^r}$ isomorphic to a proper subdivision of $T_{2,r}$. Also, there is an induced subgraph of the line graph of $W_{2^r\times 2^r}$ isomorphic to the line graph of a proper subdivision of $T_{2,r}$. Consequently, if $W$ is a subdivision of $W_{2^r\times 2^r}$, then both $W$ and its line graph have an induced minor isomorphic to $T_{2,r}$.
        \item\label{obs:trees_b} There is an induced minor of $T_{2,dr}$ isomorphic to $T_{2^d,r}$.
    \end{enumerate}
\end{observation}
\begin{figure}
    \centering
    \includegraphics[width=0.4\linewidth]{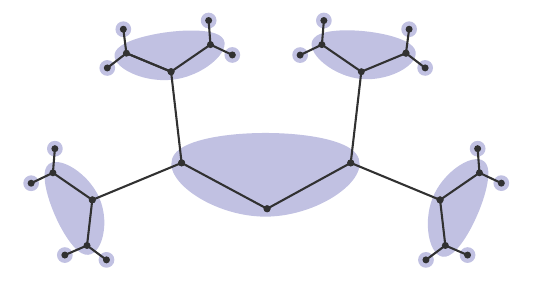}
    \caption{An induced $T_{4,2}$-model in $T_{2,4}$ (as in \ref{obs:trees}\ref{obs:trees_b} for $d=r=2$).}
    \label{fig:binaryindm}
\end{figure}

We will also use the following from \cite{completeminorpw}:

\begin{lemma}[Hickingbotham \cite{completeminorpw}]\label{lem:binaryisg}
    For every $r\in \poi$, if $G$ is a graph with an induced minor isomorphic to $T_{2,8r}$, then $G$ has an induced subgraph isomorphic to either a subdivision of $T_{2,r}$ or the line graph of a subdivision of $T_{2,r}$.
\end{lemma}

Let us now deduce Theorem~\ref{thm:pwisg}:

\begin{proof}[Proof of Theorem~\ref{thm:pwisg}]
Let 
$$\phi=f_{\ref{thm:motherKtt}}(d,l,l',2^{2r},s,s');$$
$$\gamma=g_{\ref{thm:motherKtt}}(d,l,l',2^{2r},s,s').$$

We claim that
$$f_{\ref{thm:pwisg}}=f_{\ref{thm:pwisg}}(d,l,l',r,s,s')=f_{\ref{thm:main_tree_indm}}(\max\{r,\phi,\gamma\},T_{2,16r})$$
satisfies the theorem. Let $G$ be a graph of pathwidth larger than $f_{\ref{thm:pwisg}}$. From Theorem~\ref{thm:main_tree_indm}, it follows that $G$ has a subgraph isomorphic to $K_{r+1}$, an induced minor isomorphic to $K_{\phi,\gamma}$, or an induced minor isomorphic to $T_{2,16r}$. In the first case, \ref{thm:pwisg}\ref{thm:pwisg_a} holds. Also, if $G$ has an induced minor isomorphic to $T_{2,16r}$, then by Lemma~\ref{lem:binaryisg}, $G$ has an induced subgraph isomorphic to either a subdivision of $T_{2,2r}$ or the line graph of a subdivision of $T_{2,2r}$, and again \ref{thm:pwisg}\ref{thm:pwisg_a} holds. Therefore, we may assume that $G$ has an induced minor isomorphic to $K_{\phi,\gamma}$. By the choice of $\phi,\gamma$, we can apply Theorem~\ref{thm:motherKtt} to $G$. Note that \ref{thm:motherKtt}\ref{thm:motherKtt_a}, along with Observation~\ref{obs:trees}\ref{obs:trees_a} (and the fact that $2^{2r}\geq r$), implies \ref{thm:pwisg}\ref{thm:pwisg_a}. Moreover, \ref{thm:motherKtt}\ref{thm:motherKtt_b} directly implies \ref{thm:pwisg}\ref{thm:pwisg_b}, and \ref{thm:motherKtt}\ref{thm:motherKtt_c} directly implies \ref{thm:pwisg}\ref{thm:pwisg_c}. This completes the proof of Theorem~\ref{thm:pwisg}. 
\end{proof}

\section{Seedlings and overview of the proof of Theorem~\ref{thm:main_tree_indm}}\label{sec:outline}

Here, we give an overview of the steps in the proof of Theorem~\ref{thm:main_tree_indm}, beginning with some definitions.

Let $G$ be a graph and let $X,Y\subseteq V(G)$. An \textit{$(X,Y)$-path in $G$} is a path $P$ in $G$ that has an end in $X$ and an end in $Y$, and, subject to this property, $V(P)$ is minimal with respect to inclusion. Equivalently, an $(X,Y)$-path $P$ in $G$ is a path in $G$ (recall that a ``path in $G$'' is always induced) with the following specifications:
\begin{itemize}
    \item $P$ has length zero and $P\subseteq X\cap Y$; or
    \item $P$ has non-zero length, one end of $P$ belongs to $X\setminus Y$, the other end of $P$ belongs to $Y\setminus X$, and we have $P^*\cap (X\cup Y)=\varnothing$.
\end{itemize}
In particular, every $(X,Y)$-path $P$ in $G$ has an end in $X$, called the \textit{$X$-end of $P$}, and an end in $Y$, called the \textit{$Y$-end of $P$}. So, the unique vertex of a zero-length $(X,Y)$-path $P$ is both the $X$-end and the $Y$-end of $P$.
\medskip

Next, we define what we call a ``seedling,'' a notion central to almost all of our proofs in this paper. Let $G$ be a graph and let $\lambda\in \poi$. A \textit{$\lambda$-seedling in $G$} is a triple $(A,\mca{L},Y)$ such that:
\begin{itemize}
\item $A$ is a path in $G$;
\item $Y\subseteq V(G)\setminus A$; and
\item $\mca{L}$ is a set of $\lambda$ pairwise disjoint $(N_G(A),Y)$-paths in $G\setminus A$.
\end{itemize}
See Figure~\ref{fig:seedling}. It follows in particular that $A\cap V(\mca{L})=\varnothing$, and for every $L\in \mca{L}$, the $N(A)$-end of $L$ is the only vertex in $L$ with a neighbor in $A$.
\begin{figure}
    \centering
\includegraphics[width=0.35\linewidth]{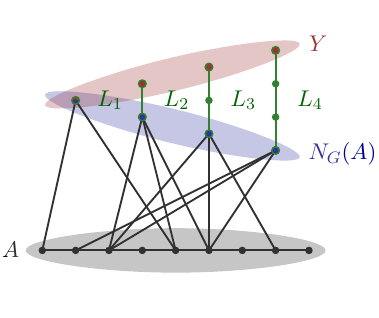}
    \caption{A $4$-seedling $(A,\{L_1,L_2,L_3,L_4\}, Y)$.}
    \label{fig:seedling}
\end{figure}

By a \textit{seedling in $G$} we mean a $\lambda$-seedling in $G$ for some $\lambda\in \poi$. Two seedlings $(A,\mca{L},Y)$ and $(A',\mca{L}',Y')$ in a graph $G$ are \textit{disjoint} if $(A\cup V(\mca{L}))\cap (A'\cup V(\mca{L}'))=\varnothing$. We say that a seedling $(A,\mca{L},Y)$ in $G$ is \textit{$\kappa$-rigid}, where $\kappa\in \poi$, if there is no set $\mca{K}$ of $\kappa$ pairwise anticomplete $(N(A),Y)$-paths in $G$ such that $V(\mca{K})\subseteq V(\mca{L})$. (Recall that $V(\mca{K})\subseteq V(\mca{L})$ means that the union of the vertex sets of the paths in $\mca{K}$ is a subset of the union of the vertex sets of the paths in $\mca{L}$, in which case $\mca{K}$ may or may not be a subset of $\mca{L}$.)

For $t\in \poi$, we say that a graph $G$ is \textit{$t$-tidy} if $G$ is $K_{t+1}$-free and has no induced minor isomorphic to $K_{t,t}$. Then Theorem~\ref{thm:main_tree_indm} says that for all $t\in \poi$ and every forest $H$, every $t$-tidy graph of sufficiently large pathwidth has an induced minor isomorphic to $H$. Roughly, the proof of Theorem~\ref{thm:main_tree_indm} is in three steps. The first step is to prove the following. Note that $g_{\ref{thm:obtain_a_seedling}}$ does not depend on $\lambda$.

\begin{restatable}{theorem}{obtainaseedling}\label{thm:obtain_a_seedling}
    For all $d,r,t,\lambda\in \poi$, there are constants $f_{\ref{thm:obtain_a_seedling}}=f_{\ref{thm:obtain_a_seedling}}(d,r,t,\lambda)\in \poi$ and $g_{\ref{thm:obtain_a_seedling}}=g_{\ref{thm:obtain_a_seedling}}(d,r,t)\in \poi$ such that for every $t$-tidy graph $G$ with $\pw(G)>f_{\ref{thm:obtain_a_seedling}}$ one of the following holds.
    \begin{enumerate}[{\rm(a)}]
    \item \label{thm:obtain_a_seedling_a} $G$ has an induced minor isomorphic to $T_{d,r}$. 
    \item \label{thm:obtain_a_seedling_b} There is a $\lambda$-seedling in $G$ which is $g_{\ref{thm:obtain_a_seedling}}$-rigid.
    \end{enumerate}
\end{restatable}

The second step is to prove the following. Note, again, that $g_{\ref{thm:seedling_branches}}$ does not depend on $\lambda$ (nor on $\delta$, but that does not matter much).

\begin{restatable}{theorem}{seedlingbranches}\label{thm:seedling_branches}
    For all $t,\delta,\lambda,\kappa\in \poi$, there are constants $f_{\ref{thm:seedling_branches}}=f_{\ref{thm:seedling_branches}}(t, \delta,\lambda,\kappa)\in \poi$ and $g_{\ref{thm:seedling_branches}}=g_{\ref{thm:seedling_branches}}(t,\kappa)\in \poi$ with the following property. Let $G$ be a $t$-tidy graph and let $(A,\mca{L},Y)$ be an $f_{\ref{thm:seedling_branches}}$-seedling in $G$ which is $\kappa$-rigid. Then there are $\delta$ pairwise disjoint $\lambda$-seedlings $(A_1,\mca{L}_1,Y_1),\ldots, (A_{\delta},\mca{L}_{\delta},Y_{\delta})$ in $G\setminus A$ with the following specifications.
    \begin{enumerate}[{\rm (a)}]
    \item \label{thm:seedling_branches_a} The paths $A_1,\ldots, A_{\delta}$ are pairwise anticomplete in $G$.
    \item \label{thm:seedling_branches_b} For every $i\in \poi_{\delta}$, we have:
    \begin{itemize}
        \item $A$ and $A_i$ are not anticomplete in $G$;
        \item $A$ and $V(\mca{L}_i)$ are anticomplete in $G$; and
        \item $(A_i,\mca{L}_i,Y_i)$ is $g_{\ref{thm:seedling_branches}}$-rigid.
    \end{itemize}
    \end{enumerate}
\end{restatable}

The third (and last) step is to use Theorem~\ref{thm:seedling_branches} to prove the following by induction on $r$:

\begin{restatable}{theorem}{seedlingtotree}\label{thm:seedling_to_tree}
     For all $d,r,t,\kappa\in \poi$, there is a constant $f_{\ref{thm:seedling_to_tree}}=f_{\ref{thm:seedling_to_tree}}(d,r,t,\kappa)\in \poi$ with the following property. Let $G$ be a $t$-tidy graph and let $(A,\mca{L},Y)$ be an $f_{\ref{thm:seedling_to_tree}}$-seedling in $G$ which is $\kappa$-rigid. Then there is an induced $T_{d,r}$-model in $G[A\cup V(\mca{L})]$ where $A$ is the branch set associated with the root of $T_{d,r}$.
\end{restatable}

Now Theorem~\ref{thm:main_tree_indm} is almost immediate from Theorems~\ref{thm:obtain_a_seedling} and \ref{thm:seedling_to_tree}:

\maintreeindm*
\begin{proof}
Let $|V(H)|=h$. Let
$$\kappa=g_{\ref{thm:obtain_a_seedling}}(h,h,t);$$
$$\lambda=f_{\ref{thm:seedling_to_tree}}(h,h,t,\kappa).$$

We claim that
$$f_{\ref{thm:main_tree_indm}}=f_{\ref{thm:main_tree_indm}}(t,H)=f_{\ref{thm:obtain_a_seedling}}(h,h,t,\lambda)$$
satisfies the theorem.

Let $G$ be a $t$-tidy graph with $\pw(G)>f_{\ref{thm:main_tree_indm}}$. We will show that $G$ has an induced minor isomorphic to $H$. Let $H^+$ be a tree obtained from $H$ by adding a vertex with exactly one neighbor in each component of $H$. Then $H$ is an induced subgraph of $H^+$. Also, $H^+$ has both maximum degree and radius at most $|V(H^+)|-1=h$. It follows that $H^+$, and so $H$, is isomorphic to an induced subgraph of $T_{h,h}$. Thus, it suffices to prove that $G$ has an induced minor isomorphic to $T_{h,h}$.

Now, since $G$ has pathwidth more than $f_{\ref{thm:obtain_a_seedling}}(h,h,t,\lambda)$, it follows from Theorem~\ref{thm:obtain_a_seedling} that either $G$ has an induced minor isomorphic to $T_{h,h}$, or there is a $\lambda$-seedling in $G$ which is $\kappa$-rigid (recall that $\kappa=g_{\ref{thm:obtain_a_seedling}}(h,h,t)$). In the former case, we are done. In the latter case, since $\lambda=f_{\ref{thm:seedling_to_tree}}(h,h,t,\kappa)$, it follows from Theorem~\ref{thm:seedling_to_tree} that $G$ has an induced minor isomorphic to $T_{h,h}$. This completes the proof of Theorem~\ref{thm:main_tree_indm}.
\end{proof}

It remains to prove Theorems~\ref{thm:obtain_a_seedling}, \ref{thm:seedling_branches}, and \ref{thm:seedling_to_tree}, which we will do in Sections~\ref{sec:plant}, \ref{sec:grow}, and \ref{sec:seedlingtotree}.

\section{Planting a seedling}\label{sec:plant}

In this section, we prove Theorem~\ref{thm:obtain_a_seedling}. We need a few results from the literature.

\begin{theorem}[Ramsey \cite{multiramsey}]\label{thm:classicalramsey}
For all $s,t\in \poi$, every graph on at least $s^t$ vertices has either a stable set of cardinality $s$ or a clique of cardinality $t+1$.
\end{theorem}

\begin{theorem}[Hickingbotham \cite{completeminorpw}]\label{thm:completeminor}
    For all $r,s\in \poi$, there is a constant $f_{\ref{thm:completeminor}}=f_{\ref{thm:completeminor}}(r,s)$ such that every graph $G$ with $\pw(G)>f_{\ref{thm:completeminor}}$ has either an induced minor isomorphic to $T_{2,r}$ or a minor isomorphic to $K_s$.
\end{theorem}

We also need Theorem~\ref{thm:noblock} below from \cite{tw7}, which we have also used in several earlier papers of this series.
\medskip

Let $X$ be a set. We denote the set of all subsets of $X$ by $2^X$ and the set of all $k$-subsets of $X$, where $k\in \poi$, by $\binom{X}{k}$. Let $k,l\in \poi$ and let $G$ be a graph. A \textit{$(k,l)$-block} in $G$ is a pair $(B, \mca{P})$ where $B\subseteq V(G)$ with $|B|\geq k$ and $\mca{P}:{B\choose 2}\rightarrow 2^{V(G)}$ is a map such that $\mca{P}_{\{x,y\}}=\mca{P}(\{x,y\})$, for each $2$-subset $\{x,y\}$ of $B$, is a set of at least $l$ pairwise internally disjoint paths in $G$ from $x$ to $y$. We say that $(B,\mca{P})$ is \textit{strong} if for all distinct $2$-subsets $\{x,y\}, \{x',y'\}$ of $B$, we have $V(\mca{P}_{\{x,y\}})\cap V(\mca{P}_{\{x',y'\}})=\{x,y\}\cap\{x',y'\}$; that is, each path $P\in \mca{P}_{\{x,y\}}$ is disjoint from each path $P'\in \mca{P}_{\{x',y'\}}$, except that $P$ and $P'$ may share an end.

Let $t\in \poi$. We say that a graph $G$ is \textit{$t$-clean} if $G$ has no induced subgraph isomorphic to $K_{t+1}$, $K_{t,t}$, a subdivision of $W_{t\times t}$, or the line graph of a subdivision of $W_{t\times t}$.

\begin{theorem}[Abrishami, Alecu, Chudnovsky, Hajebi, Spirkl \cite{tw7}]\label{thm:noblock}
For all $k,l,t\in \poi$, there is a constant $f_{\ref{thm:noblock}}=f_{\ref{thm:noblock}}(k,l,t)\in \poi$ such that for every $t$-clean graph $G$ with $\tw(G)>f_{\ref{thm:noblock}}$, there is a strong $(k,l)$-block in $G$.
\end{theorem}

Finally, we need a lemma from \cite{tw17}:

\begin{lemma}[Chudnovsky, Hajebi, Spirkl \cite{tw17}]\label{lem:comp_model_rigid} For all $s,t,\rho,\sigma\in \poi$, there are constants $f_{\ref{lem:comp_model_rigid}}=f_{\ref{lem:comp_model_rigid}}(s,t,\rho,\sigma)\in \poi$ and $g_{\ref{lem:comp_model_rigid}}=g_{\ref{lem:comp_model_rigid}}(s,\rho,\sigma)\in \poi$ with the following property. Let $G$ be a $K_{t+1}$-free graph and let $(B,\mca{Q})$ be a strong $(f_{\ref{lem:comp_model_rigid}},g_{\ref{lem:comp_model_rigid}})$-block in $G$ such that for every $\{x,y\}\subseteq B$, the paths $(Q^*: Q\in \mca{Q}_{\{x,y\}})$ are pairwise anticomplete in $G$. Then one of the following holds.
\begin{enumerate}[{\rm (a)}]
 \item\label{lem:comp_model_rigid_a} There is an induced subgraph of $G$ isomorphic to a proper subdivision of $K_s$.
    \item \label{lem:comp_model_rigid_b} There is an induced minor of $G$ isomorphic to $K_{\rho,\sigma}$.
\end{enumerate}
\end{lemma}

Now we can prove the main result of this section, which we restate:

\obtainaseedling*

\begin{proof}
Let 
$$\phi=\phi(d,r,t)=f_{\ref{lem:comp_model_rigid}}(rd^r+1,t,t,t);$$
$$\psi=\psi(d,r,t,\lambda)=f_{\ref{thm:noblock}}(\phi^t,\lambda,\max\{2^{dr},t\}).$$

We claim that
$$f_{\ref{thm:obtain_a_seedling}}=f_{\ref{thm:obtain_a_seedling}}(d,r,t,\lambda)=f_{\ref{thm:completeminor}}(dr,\psi+2)$$
and
$$g_{\ref{thm:obtain_a_seedling}}=g_{\ref{thm:obtain_a_seedling}}(d,r,t)=g_{\ref{lem:comp_model_rigid}}(rd^r+1,t,t)$$
satisfy the theorem.

Let $G$ be a $t$-tidy graph with $\pw(G)>f_{\ref{thm:obtain_a_seedling}}$. Suppose for a contradiction that neither \ref{thm:obtain_a_seedling}\ref{thm:obtain_a_seedling_a} nor \ref{thm:obtain_a_seedling}\ref{thm:obtain_a_seedling_b} holds; that is, $G$ has no induced minor isomorphic to $T_{d,r}$, and there is no $\lambda$-seedling in $G$ which is $g_{\ref{thm:obtain_a_seedling}}$-rigid.

\sta{\label{st:Gisclean} $G$ has no induced minor isomorphic to $T_{2,dr}$. Also, $G$ is $\max\{2^{dr},t\}$-clean.}

Since $G$ has no induced minor isomorphic to $T_{d,r}$, it follows from Observation~\ref{obs:trees}\ref{obs:trees_b} (and the fact that $2^d>d$) that $G$ has no induced minor isomorphic to $T_{2,dr}$. This, along with Observation~\ref{obs:trees}\ref{obs:trees_a}, implies that $G$ has no induced subgraph isomorphic to a subdivision of $W_{2^{dr}\times 2^{dr}}$ or the line graph of a subdivision of $W_{2^{dr}\times 2^{dr}}$. Also, recall that $G$ is $K_{t+1}$-free and $K_{t,t}$-free. Thus, $G$ is $\max\{2^{dr},t\}$-clean. This proves \eqref{st:Gisclean}.
\medskip

Since $\pw(G)>f_{\ref{thm:obtain_a_seedling}}$, it follows from Theorem~\ref{thm:completeminor}, the choice of $f_{\ref{thm:obtain_a_seedling}}$, and the first assertion in \eqref{st:Gisclean} that $G$ has a minor isomorphic to $K_{\psi+2}$; in particular, we have $\tw(G)\geq \tw(K_{\psi+2})=\psi+1$ \cite{diestel}. Moreover, by the second assertion in \eqref{st:Gisclean}, $G$ is $\max\{2^{dr},t\}$-clean. Thus, by Theorem~\ref{thm:noblock} and the choice of $\psi$, there is a strong $(\phi^t,\lambda)$-block $(B,\mca{P})$ in $G$. Since $G$ is $K_{t+1}$-free, it follows from Theorem~\ref{thm:classicalramsey} that there is a stable set $S\subseteq B$ in $G$ with $|S|=\phi$. We further claim that:

\sta{\label{st:antiint1} For every $\{x,y\}\subseteq S$, there is a set $\mca{Q}_{\{x,y\}}$ of $g_{\ref{thm:obtain_a_seedling}}$ paths in $G$ between $x$ and $y$ such that $V(\mca{Q}_{\{x,y\}})\subseteq V(\mca{P}_{\{x,y\}})$ and the paths $(Q^*:Q\in \mca{Q}_{\{x,y\}})$ are pairwise anticomplete in $G$.}

Since $S$ is a stable set, we have $xy\notin E(G)$ (thus, $N_{G}(y)\subseteq V(G)\setminus \{x\}$), and the paths in $\mca{P}_{\{x,y\}}$ have non-empty interiors; in particular, $\mca{L}_{\{x,y\}}=\{P^*:P\in \mca{P}_{\{x,y\}}\}$ is a set of $\lambda$ pairwise disjoint $(N_G(x),N_G(y))$-paths in $G$. Thus, $(\{x\},\mca{L}_{\{x,y\}},N_{G}(y))$ is a $\lambda$-seedling in $G$. On the other hand, recall the assumption that there is no $\lambda$-seedling in $G$ which is $g_{\ref{thm:obtain_a_seedling}}$-rigid. It follows that $(\{x\},\mca{L}_{\{x,y\}},N_{G}(y))$ is not $g_{\ref{thm:obtain_a_seedling}}$-rigid, and so there is a set $\mca{K}_{\{x,y\}}$ of $g_{\ref{thm:obtain_a_seedling}}$ pairwise disjoint and anticomplete $(N_{G}(x),N_G(y))$-paths in $G$ with $V(\mca{K}_{\{x,y\}})\subseteq V(\mca{L}_{\{x,y\}})$. But now $\mca{Q}_{\{x,y\}}=\{\{x,y\}\cup K:K\in \mca{K}_{\{x,y\}}\}$ is a set of $g_{\ref{thm:obtain_a_seedling}}$ paths in $G$ between $x,y$, such that $V(\mca{Q}_{\{x,y\}})\subseteq V(\mca{P}_{\{x,y\}})$, and $(Q^*:Q\in \mca{Q}_{\{x,y\}})$ are pairwise anticomplete in $G$. This proves \eqref{st:antiint1}.
\medskip

Henceforth, for every $2$-subset $\{x,y\}$ of $S$, let $\mca{Q}_{\{x,y\}}$ be as given by \eqref{st:antiint1}. Then $(S,\mca{Q})$ is a strong $(\phi, g_{\ref{thm:obtain_a_seedling}})$-block in $G$ such that for every $\{x,y\}\subseteq S$, the paths $(Q^*:Q\in \mca{Q}_{\{x,y\}})$ are pairwise anticomplete in $G$. Since $G$ is $K_{t+1}$-free with no induced minor isomorphic to $K_{t,t}$, it follows from Lemma~\ref{lem:comp_model_rigid} and the choice of $\phi$ and $g_{\ref{thm:obtain_a_seedling}}$ that $G$ has an induced subgraph isomorphic to a proper subdivision of $K_{rd^r+1}$. In particular, since $|V(T_{d,r})|\leq rd^r+1$, it follows that $G$ has an induced subgraph isomorphic to a (proper) subdivision of $T_{d,r}$. But then $G$ has an induced minor isomorphic to $T_{d,r}$, contrary to the assumption that \ref{thm:obtain_a_seedling}\ref{thm:obtain_a_seedling_a} does not hold. This completes the proof of Theorem~\ref{thm:obtain_a_seedling}.
\end{proof}

\section{Growing a seedling}\label{sec:grow}

In this section, we prove Theorem~\ref{thm:seedling_branches}. The main tool is Lemma~\ref{lem:digraph} below about digraphs. So we start by clarifying our digraph terminology.

By a \textit{digraph} we mean a pair $D=(V(D), E(D))$ where $V(D)$ is a finite set of \textit{vertices} and $E(D)\subseteq (V(D)\times V(D))\setminus \{(v,v):v\in V(D)\}$ is the set of \textit{edges}. In particular, our digraphs are loopless and allow at most one edge in each direction between every two vertices.
Let $D$ be a digraph. For $(u,v)\in E(D)$, we say that $v$ is an \textit{out-neighbor} of $u$ and $u$ is an \textit{in-neighbor} of $v$. The \textit{out-degree} (\textit{in-degree}) of a vertex $v\in V(D)$ is the number of its out-neighbors (in-neighbors). The \textit{underlying graph of $D$} is the graph $G$ with $V(G)=V(D)$ and $E(G)=\{uv:(u,v)\in E(D) \text{ or }(v,u)\in E(D)\}$. A \textit{stable set in $D$} is a stable set in the underlying graph of $D$. For $X\subseteq V(D)$, we denote by $D[X]$ the digraph with vertex set $X$ and edge set $E(D)\cap (X\times X)$. It follows that the underlying graph of $D[X]$ is the subgraph of the underlying graph of $D$ induced by $X$.

We need the following lemma; \ref{lem:digraph}\ref{lem:digraph_a} is well known, but we include a proof for completeness. 

\begin{lemma}\label{lem:digraph}
    Let $q,r,s\in \poi$ and let $D$ be a digraph. Then the following hold.
    \begin{enumerate}[{\rm (a)}]
    \item \label{lem:digraph_a} 
    If $D$ has at least $2rs$ vertices of out-degree at most $r$, then there is a stable set of cardinality $s$ in $D$.
    \item \label{lem:digraph_b} If there are at least $2qrs$ vertices of out-degree at least $qr$ in $D$, then there is an $s$-subset $S$ of $V(D)$ with the following property: for every $q$-subset $\{v_1,\ldots, v_q\}$ of $S$, there are $q$ pairwise disjoint $r$-subsets $R_1,\ldots, R_q$ of $V(D)\setminus S$ such that for each $i\in \poi_{q}$, every vertex in $R_i$ is an out-neighbor of $v_i$.
    \end{enumerate}
\end{lemma}
\begin{proof}
    We prove \ref{lem:digraph}\ref{lem:digraph_a} by induction on $s$ (for fixed $r$). The case $s=1$ is trivial, so assume that $s\geq 2$. Choose a set $X$ of exactly $2rs$ vertices, each with out-degree at most $r$ in $D$. Let $D_1=D[X]$. Then $D_1$ is a digraph in which every vertex has out-degree at most $r$. Let $G_1$ be the underlying graph of $D_1$. It follows that both $D_1$ and $G_1$ have at most $2r^2s$ edges. There are two cases to consider. First, assume that some vertex $v$ has degree at most $2r-1$ in $G_1$. Let $D_2=D[X\setminus (N_{G_1}(v)\cup \{v\})]$. Then we have $|V(D_2)|=|X|-|N_{G_1}(v)\cup \{v\}|\geq 2rs-2r\geq 2r(s-1)$, and every vertex of $D_2$ has out-degree at most $r$ in $D_2$. By the inductive hypothesis applied to $D_2$, there is a stable set $S$ in $D_2$ of cardinality $s-1$. But now, since $S\subseteq V(D_2)=X\setminus (N_{G_1}(v)\cup \{v\})$, it follows that $S\cup \{v\}$ is a stable set of cardinality $s$ in $D$, as desired. Second, assume that every vertex in $G_1$ has degree at least $2r$. Then, since $G_1$ has $2rs$ vertices and at most $2r^2s$ edges, it follows that $G_1$ is a $2r$-regular graph (on $2rs$ vertices), and so by Brooks' theorem \cite{brooks}, $G_1$ admits a $2r$-coloring. Therefore, there is a stable set of cardinality $2rs/2r=s$ in $G_1$, and so in $D$. This proves \ref{lem:digraph}\ref{lem:digraph_a}.
    \medskip

Next, we prove \ref{lem:digraph}\ref{lem:digraph_b}, and for that we will use \ref{lem:digraph}\ref{lem:digraph_a}, which we just proved above. Let $Y$ be the set of all vertices of out-degree at least $qr$ in $D$; thus, $|Y|\geq 2qrs$. For each vertex $y\in Y$, choose a set $Q_y$ of exactly $qr$ out-neighbors of $y$ in $D$. Let $D'$ be the digraph with $V(D')=V(D)$ and $E(D')=\bigcup_{y\in Y}\{(y,z):z\in Q_y\}$. Then $E(D')\subseteq E(D)$. Moreover, for every $y\in Y$, the set $Q_y$ is exactly the set of all out-neighbors of $y$ in $D'$. In particular, $y$ has out-degree exactly $qr$ in $D'$, and so $y$ has out-degree at most $qr$ in $D'[Y]$. Since $|Y|\geq 2qrs$, it follows from \ref{lem:digraph}\ref{lem:digraph_a} applied to $D'[Y]$ that there is a stable set $S$ in $D'[Y]$ of cardinality $s$. In other words, $S$ is a stable set in $D'$ with $|S|=s$ and $S\subseteq Y$. From this and the definition of $D'$, we deduce that for every $y\in S\subseteq Y$, we have $Q_y\subseteq V(D')\setminus S=V(D)\setminus S$.

Now, let $\{v_1,\ldots, v_q\}$ be a $q$-subset of $S\subseteq Y$. Since $|Q_{v_1}|=\cdots=|Q_{v_q}|=qr$, it follows that for every $i\in \poi_{q-1}$, we have $|Q_{v_{i+1}}|-ir=(q-i)r\geq r$, and so there are $r$-subsets $R_1,\ldots, R_q$ of $Q_{v_1},\ldots, Q_{v_q}$, respectively, such that for every $i\in \poi_{q-1}$, we have $R_{i+1}\subseteq Q_{v_{i+1}}\setminus (R_1\cup \cdots \cup R_{i})$. It follows that $R_1,\ldots, R_q$ are pairwise disjoint. Moreover, for every $i\in \poi_{q}$, since $R_i$ is an $r$-subset of $Q_{v_i}$, it follows that $R_i$ is an $r$-subset of $V(D)\setminus S$ and every vertex in $R_i$ is an out-neighbor of $v_i$. This completes the proof of Lemma~\ref{lem:digraph}.
\end{proof}

The following lemma involves several applications of Lemma~\ref{lem:digraph}, and is the heart of the proof of Theorem~\ref{thm:seedling_branches} (see Figure~\ref{fig:magiclemma}).

\begin{lemma}\label{lem:magic}
Let $t,\delta,\lambda\in \poi$, let $G$ be a $K_{t+1}$-free graph and let $\mca{L}_0$ be a set of pairwise disjoint paths in $G$ with
\[|\mca{L}_0|=\left(10\delta^{t+3}\lambda^{3}\right)^t\]
such that no two paths $L,L'\in \mca{L}_0$ are anticomplete in $G$. For each $L\in \mca{L}_0$, let $x_L,y_L$ be the ends of $L$ (where $x_L=y_L$ is possible). Then there are $\delta$ paths $L_1,\ldots,L_{\delta}\in \mca{L}_0$ along with a vertex $z_{L_i}\in L_i$ for each $i\in \poi_{\delta}$, as well as $\delta$ pairwise disjoint $\lambda$-subsets $\mca{L}_1,\ldots, \mca{L}_{\delta}$ of $\mca{L}_0\setminus \{L_1,\ldots,L_{\delta}\}$, such that the following hold. 
    \begin{enumerate}[{\rm (a)}]
    \item \label{lem:magic_a} The paths $(x_{L_i}\dd L_i\dd z_{L_i}: i\in \poi_{\delta})$ are pairwise anticomplete in $G$.
    \item \label{lem:magic_b} For each $i\in \poi_{\delta}$, every path $L\in \mca{L}_i$ contains a vertex $w_L$ distinct from $x_L$ such that $w_L$ is the only vertex in $w_L\dd L\dd y_L$ with a neighbor in $x_{L_i}\dd L_i\dd z_{L_i}$.
    \end{enumerate}
\end{lemma}
\begin{figure}[t!]
    \centering
    \includegraphics[width=0.7\linewidth]{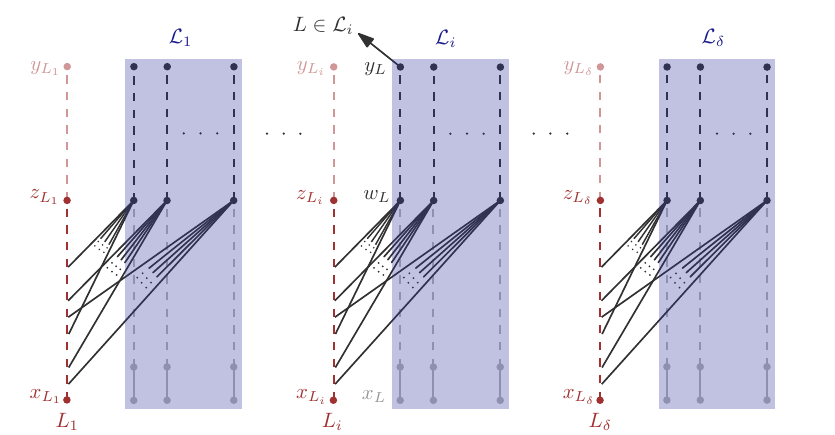}
    \caption{Lemma~\ref{lem:magic}. Dashed lines represent paths of arbitrary length (possibly zero).}
    \label{fig:magiclemma}
\end{figure}

\begin{proof}
 Since $G$ is $K_{t+1}$-free, it follows from Theorem~\ref{thm:classicalramsey} that there exists $\mca{K}_1\subseteq \mca{L}_0$ with 
 $$|\mca{K}_1|= 10\delta^{t+3}\lambda^3$$
 such that $\{x_L:L\in \mca{K}_1\}$ is a stable set in $G$. Let $D_1$ be the digraph with $V(D_1)=\mca{K}_1$ such that for distinct $L,L'\in \mca{K}_1$, we have $(L,L')\in E(D_1)$ if and only if $x_L$ has a neighbor in $L'$.
 \medskip

 Suppose that $D_1$ has at least $2\delta^2\lambda$ vertices of out-degree at least $\delta\lambda$. Applying Lemma~\ref{lem:digraph}\ref{lem:digraph_b} to $D_1$ (with $r=\lambda$ and $q=s=\delta$), we deduce that there is a $\delta$-subset $\{L_1,\ldots,L_{\delta}\}$ of $\mca{K}_1\subseteq \mca{L}_0$ as well as $\delta$ pairwise disjoint $\lambda$-subsets $\mca{L}_1,\ldots, \mca{L}_{\delta}$ of $\mca{K}_1\setminus \{L_1,\ldots,L_{\delta}\}\subseteq \mca{L}_0\setminus \{L_1,\ldots,L_{\delta}\}$ such that for each $i\in \poi_{\delta}$, the vertex $x_{L_i}$ has neighbors in every path in $\mca{L}_{i}$. Moreover, since $\{x_L:L\in \mca{K}_1\}$ is a stable set in $G$, it follows that: 
\begin{itemize}
    \item $\{x_{L_i}: i\in \poi_{\delta}\}$ is a stable set in $G$; and
    \item for each $i\in \poi_{\delta}$ and every $L\in \mca{L}_i$, traversing $L$ from $y_L$ to $x_L$, the first neighbor $w_L$ of $x_{L_i}$ in $L$ is distinct from $x_L$. In particular, $w_L$ is the only neighbor of $x_{L_i}$ in $w_L\dd L\dd y_L$.
    \end{itemize}
But now we are done by setting $z_{L_i}=x_{L_i}$ for every $i\in \poi_{\delta}$.
\medskip

Henceforth, assume that there are at most $2\delta^2\lambda$ vertices of out-degree at least $\delta\lambda$ in $D_1$. Since $2\delta^2\lambda\leq 2\delta^{t+3}\lambda^3$, it follows that there are at least $|V(D_1)|-2\delta^{t+3}\lambda^3=8\delta^{t+3}\lambda^3$ vertices of out-degree at most $\delta\lambda$ in $D_1$. Thus, applying Lemma~\ref{lem:digraph}\ref{lem:digraph_a} to $D_1$ (with $r=\delta\lambda$ and $s=4\delta^{t+2}\lambda^2$), it follows that there is a stable set $\mca{K}_2\subseteq \mca{K}_1=V(D_1)$ in $D_1$ with
$$|\mca{K}_2|=4\delta^{t+2}\lambda^2.$$

Specifically, we have:

\sta{\label{st:endsanti} For all distinct $L,L'\in \mca{K}_2$, the end $x_{L}$ of $L$ is anticomplete to $L'$ in $G$.}

Recall also that no two paths $L, L'\in \mca{K}_2\subseteq \mca{L}_0$ are anticomplete in $G$. In particular, since $|\mca{K}_2|=4\delta^{t+2}\lambda^2>\delta\lambda$, it follows that for every $L\in \mca{K}_2$, there are at least $\delta\lambda$ paths $L'\in \mca{K}_2\setminus \{L\}$ such that $L,L'$ are not anticomplete in $G$. This, combined with \eqref{st:endsanti}, implies that:

\sta{\label{st:thetrick} For every $L\in \mca{K}_2$, there are distinct and adjacent vertices $z^-_{L},z_L\in L$ such that $L$ traverses $x_L,z^-_L,z_L,y_L$ in order (where $x_L=z^-_L$ and $z_L=y_L$ are both possible), and the following hold.
\begin{itemize}
   \raggedright \item There are at least $\delta\lambda$ paths $L'\in \mca{K}_2\setminus \{L\}$ for which $x_L\dd L\dd z_L$ and $L'$ are not anticomplete in $G$.
    \item There are at most $\delta\lambda$ paths $L'\in \mca{K}_2\setminus \{L\}$ for which $x_L\dd L\dd z^-_L$ and $L'$ are not anticomplete in $G$.
\end{itemize}}

Next, let $D_2$ be the digraph with $V(D_2)=\mca{K}_2$ such that for all distinct $L,L'\in \mca{K}_2$, we have $(L,L')\in E(D_2)$ if and only if $x_L\dd L\dd z_L$ and $L'$ are not anticomplete in $G$.
\medskip

By the first bullet of \eqref{st:thetrick}, every vertex has out-degree at least $\delta\lambda$ in $D_2$. Recall also that $|V(D_2)|=|\mca{K}_2|=4\delta^{t+2}\lambda^2$. Thus, applying Lemma~\ref{lem:digraph}\ref{lem:digraph_b} to $D_2$ (with $q=\delta$, $r=\lambda$ and $s=2\delta^{t+1}\lambda$), it follows that:

\sta{\label{st:manyseemany} There is a $(2\delta^{t+1}\lambda)$-subset $\mca{K}_3$ of $\mca{K}_2=V(D_2)$ with the following property: for every $\delta$-subset $\{L_1,\ldots,L_{\delta}\}$ of $\mca{K}_3$, there are $\delta$ pairwise disjoint $\lambda$-subsets $\mca{L}_1,\ldots, \mca{L}_{\delta}$ of $\mca{K}_2\setminus \mca{K}_3$ such that for all $i\in \poi_{\delta}$ and $L\in \mca{L}_{i}$, the paths $x_{L_i}\dd L_i\dd z_{L_i}$ and $L$ are not anticomplete in $G$.} 

From now on, let $\mca{K}_3$ be as given by \eqref{st:manyseemany}. We claim that:

\sta{\label{st:shortestanti} There is a $\delta$-subset $\{L_1,\ldots,L_{\delta}\}$ of $\mca{K}_3$ for which $(x_{L_i}\dd L_i\dd z_{L_i}: i\in \poi_{\delta})$ are pairwise anticomplete in $G$.}

To see this, let $D_3$ be the digraph with $V(D_3)=\mca{K}_3$ such that for all distinct $L,L'\in \mca{K}_3$, we have $(L,L')\in E(D_3)$ if and only if $x_L\dd L\dd z^-_L$ and $L'$ are not anticomplete in $G$. Since $\mca{K}_3\subseteq \mca{K}_2$, it follows from the second bullet of \eqref{st:thetrick} that every vertex in $D_3$ has out-degree at most $\delta\lambda$. Recall also that by \eqref{st:manyseemany}, we have $|V(D_3)|=|\mca{K}_3|=2\delta^{t+1}\lambda$. Thus, applying Lemma~\ref{lem:digraph}\ref{lem:digraph_a} to $D_3$ (with $r=\delta\lambda$ and $s=\delta^t$), we deduce that there is a stable set $\mca{K}_4\subseteq \mca{K}_3=V(D_3)$ of cardinality $\delta^t$ in $D_3$. It follows that for all distinct $L,L'\in \mca{K}_4$, the paths $x_L\dd L\dd z^-_L$ and $L'$ are anticomplete in $G$; in particular, $x_L\dd L\dd z^-_L$ and $x_{L'}\dd L'\dd z_{L'}$ are anticomplete in $G$. Also, since $|\mca{K}_4|=\delta^t$ and since $G$ is $K_{t+1}$-free, it follows from Theorem~\ref{thm:classicalramsey} that there are $\delta$ paths $L_1,\ldots, L_{\delta}\in \mca{K}_4\subseteq \mca{K}_3$ for which $\{z_{L_i}:i\in \poi_{\delta}\}$ is a stable set in $G$. But now $(x_{L_i}\dd L_i\dd z_{L_i}: i\in \poi_{\delta})$ are pairwise anticomplete in $G$. This proves \eqref{st:shortestanti}.
\medskip

We can now finish the proof. Let $\{L_1,\ldots,L_{\delta}\}$ be the $\delta$-subset of $\mca{K}_3$ given by \eqref{st:shortestanti}. By \eqref{st:manyseemany}, there are $\delta$ pairwise disjoint $\lambda$-subsets $\mca{L}_1,\ldots, \mca{L}_{\delta}$ of $\mca{K}_2\setminus \mca{K}_3\subseteq \mca{K}_2\setminus \{L_1,\ldots,L_{\delta}\}\subseteq \mca{L}_0\setminus \{L_1,\ldots,L_{\delta}\}$, such that for all $i\in \poi_{\delta}$ and $L\in \mca{L}_{i}$, the paths $x_{L_i}\dd L_i\dd z_{L_i}$ and $L$ are not anticomplete in $G$. Moreover, 
\begin{itemize}
    \item By \eqref{st:shortestanti}, the paths $(x_{L_i}\dd L_i\dd z_{L_i}: i\in \poi_{\delta})$ are pairwise anticomplete in $G$.
    \item For all $i\in \poi_{\delta}$ and $L\in \mca{L}_i$, since $L_i,L\in \mca{K}_2\subseteq \mca{K}_1$, it follows from \eqref{st:endsanti} that, traversing $L$ from $y_L$ to $x_L$, the first vertex $w_L$ in $L$ with a neighbor in $x_{L_i}\dd L_i\dd z_{L_i}$ is distinct from $x_L$. In particular, $w_L$ is the only vertex in $w_L\dd L\dd y_L$ with a neighbor in $x_{L_i}\dd L_i\dd z_{L_i}$.
    \end{itemize}

But now $(L_i,z_{L_i},\mca{L}_i:i\in \poi_{\delta})$ satisfy \ref{lem:magic}\ref{lem:magic_a} and \ref{lem:magic}\ref{lem:magic_b}. This completes the proof of Lemma~\ref{lem:magic}.
\end{proof}

We need one more lemma. The proof relies on the product version of Ramsey's theorem:

\begin{theorem}[Graham, Rothschild, Spencer \cite{productramsey}]\label{thm:productramsey}
For all $n,q,r\in \poi$, there is a constant $f_{\ref{thm:productramsey}}=f_{\ref{thm:productramsey}}(n,q,r)\in \poi$ with the following property. Let $U_1,\ldots, U_n$ be $n$ sets, each of cardinality at least $f_{\ref{thm:productramsey}}$, and let $W$ be a non-empty set of cardinality at most $r$. Let $\Phi$ be a map from the Cartesian product $U_1\times \cdots \times U_n$ to $W$. Then there exist $i\in W$ and a $q$-subset $Z_j$ of $U_j$ for each $j\in \poi_{n}$, such that for every $z\in Z_1\times \cdots\times Z_n$, we have $\Phi(z)=i$.
\end{theorem}

We will use the following both here and in the next section:

\begin{lemma}\label{lem:bigramsey}
    For all $r,s,t\in \poi$, there is a constant $f_{\ref{lem:bigramsey}}=f_{\ref{lem:bigramsey}}(r,s,t)\in \poi$ with the following property. Let $G$ be a graph with no induced minor isomorphic to $K_{t,t}$ and let $\mca{U}$ be a set of pairwise disjoint connected induced subgraphs of $G$. Assume that there is a $2rt$-subset $\mca{A}$ of $\mca{U}$ as well as $2rt$ pairwise disjoint $f_{\ref{lem:bigramsey}}$-subsets $(\mca{B}_U:U\in \mca{A})$ of $\mca{U}\setminus \mca{A}$ such that
    
    \begin{itemize}
        \item the sets in $\mca{A}$ are pairwise anticomplete in $G$; and
        \item for every $U\in \mca{A}$, the sets in $\mca{B}_{U}$ are pairwise anticomplete in $G$.
    \end{itemize}
    
Then there are $A_1,\ldots,A_{r}\in \mca{A}$ along with an $s$-subset $\mca{B}_i$ of $\mca{B}_{A_i}$ for each $i\in \poi_{r}$, such that $(A_i\cup V(\mca{B}_i):i\in \poi_{r})$ are pairwise anticomplete in $G$. In particular, we have $f_{\ref{lem:bigramsey}}\geq s$.
\end{lemma}

\begin{proof}
    Let
$$f_{\ref{lem:bigramsey}}=f_{\ref{lem:bigramsey}}(r,s,t)=f_{\ref{thm:productramsey}}\left(2rt,\max\{s,t\},2^{4r^2t^2+\binom{2rt}{2}}\right).$$
Fix an enumeration $\mca{A}=\{U_1,\ldots, U_{2rt}\}$. For every $z=(B_1,\ldots, B_{2rt})\in \mca{B}_{U_1}\times \cdots\times \mca{B}_{U_{2rt}}$,
\begin{itemize}
    \item let $E_{z}$ be the set of all ordered pairs $(i,j)\in \poi_{2rt}\times \poi_{2rt}$ with $i\neq j$ for which $B_i$ and $U_j$ are not anticomplete in $G$; and
    \item let $E'_{z}$ be the set of all $2$-subsets $\{i,j\}$ of $\poi_{2rt}$ for which $B_i$ and $B_j$ are not anticomplete in $G$.
\end{itemize}

It follows that the function $\Phi:\mca{B}_{U_1}\times \cdots\times \mca{B}_{U_{2rt}}\rightarrow 2^{\poi_{2rt}\times \poi_{2rt}}\times 2^{\binom{\poi_{2rt}}{2}}$ with $\Phi(z)=(E_{z},E'_{z})$ is well-defined. By Theorem~\ref{thm:productramsey} and the choice of $f_{\ref{lem:bigramsey}}$, we obtain $E\subseteq \poi_{2rt}\times \poi_{2rt}$ and $E'\subseteq \binom{\poi_{2rt}}{2}$, as well as a $\max\{s,t\}$-subset $\mca{B}'_{U_i}$ of $\mca{B}_{U_i}$ for each $i\in \poi_{2rt}$, such that for every $z\in \mca{B}'_{U_1}\times \cdots\times \mca{B}'_{U_{2rt}}$, we have $\Phi(z)=(E,E')$. In particular, for all $(i,j)\in E$, we have $i\neq j$.

Let $D$ be the digraph with vertex set $\poi_{2rt}$ and edge set $E$. We claim that:

\sta{\label{st:Eempty}Every vertex in $D$ has out-degree less than $t$.}

Suppose not. Then there are $i,j_1,\ldots,j_t\in \poi_{2rt}$ such that $(i,j_1),\ldots, (i,j_t)\in E$. Since $|\mca{B}'_{U_i}|=\max\{s,t\}$, we may choose $t$ distinct sets $X_1,\ldots, X_t\in \mca{B}'_{U_i}$. It follows that $X_1,\ldots, X_t\in \mca{B}'_{U_i}\subseteq \mca{B}_{U_i}$ are pairwise anticomplete in $G$. Recall also that $U_{j_1},\ldots, U_{j_t}\in \mca{A}$ are pairwise anticomplete in $G$. Moreover, for every $z\in \mca{B}'_{U_1}\times \cdots\times \mca{B}'_{U_{2rt}}$, since $\Phi(z)=(E,E')$, it follows that $E_z=E$, and so $(i,j_1),\ldots, (i,j_t)\in E_z$. In particular, for all $k,l\in \poi_{t}$, the sets $X_k,U_{j_l}$ are not anticomplete in $G$. But now $(X_1,\ldots, X_t,U_{j_1},\ldots, U_{j_t})$ is an induced $K_{t,t}$-model in $G$, a contradiction. This proves \eqref{st:Eempty}.

\sta{\label{st:E'empty}$E'=\varnothing$.}

Suppose that some $\{i,j\}\subseteq \poi_{2rt}$ belongs to $E'$. Since $|\mca{B}'_{U_i}|=|\mca{B}'_{U_j}|=\max\{s,t\}$, we may choose $t$ distinct sets $X_1,\ldots, X_t\in \mca{B}'_{U_i}$ and $t$ distinct sets $Y_1,\ldots, Y_t\in \mca{B}'_{U_j}$. It follows that $X_1,\ldots, X_t\in \mca{B}'_{U_i}\subseteq \mca{B}_{U_i}$ are pairwise anticomplete in $G$, and so are $Y_1,\ldots, Y_t\in \mca{B}'_{U_j}\subseteq \mca{B}_{U_j}$. Also, for every $z\in \mca{B}'_{U_1}\times \cdots\times \mca{B}'_{U_{2rt}}$, since $\Phi(z)=(E,E')$, it follows that $E'_z=E'$, and so $\{i,j\}\in E'_z$. In particular, for all $k,l\in \poi_{t}$, the sets $X_k,Y_{l}$ are not anticomplete in $G$. But now $(X_1,\ldots, X_t,Y_1,\ldots, Y_t)$ is an induced $K_{t,t}$-model in $G$, a contradiction. This proves \eqref{st:E'empty}.
\medskip

Since $|V(D)|=2rt$, it follows from \eqref{st:Eempty} and Lemma~\ref{lem:digraph}\ref{lem:digraph_a} that $D$ contains a stable set $\{k_1,\ldots,k_{r}\}\subseteq \poi_{2rt}$. For every $i\in \poi_{r}$, let $A_i=U_{k_i}$ and choose an $s$-subset $\mca{B}_i$ of $\mca{B}'_{U_{k_i}}\subseteq \mca{B}_{U_{k_i}}$ (this is possible because $|\mca{B}'_{U_{k_i}}|=\max\{s,t\}$). Recall that $A_1,\ldots, A_{r}\in \mca{A}$ are pairwise anticomplete in $G$. Since $\{k_1,\ldots,k_{r}\}$ is a stable set in $D$, it follows that for all distinct $i,j\in \poi_{r}$, we have $(k_i,k_j)\notin E$, and so $A_j$ and $V(\mca{B}_i)$ are anticomplete in $G$. Also, by \eqref{st:E'empty}, the sets $V(\mca{B}_1), \ldots, V(\mca{B}_{r})$ are pairwise anticomplete in $G$. But now $(A_i\cup V(\mca{B}_i):i\in \poi_{r})$ are pairwise anticomplete in $G$. This completes the proof of Lemma~\ref{lem:bigramsey}.
\end{proof}

We can now prove the main result of this section, which we restate:

\seedlingbranches*

\begin{proof}
Let
$$f_{\ref{thm:seedling_branches}}=f_{\ref{thm:seedling_branches}}(t,\delta,\lambda,\kappa)=\kappa^{\displaystyle\left(10(\delta+3t\kappa)^{t+3}\lambda^3\right)^t};$$
and let
$$g_{\ref{thm:seedling_branches}}=g_{\ref{thm:seedling_branches}}(t,\kappa)=f_{\ref{lem:bigramsey}}(\kappa,1,t).$$

Let $G$ be a $t$-tidy graph and let $(A,\mca{L},Y)$ be an $f_{\ref{thm:seedling_branches}}$-seedling in $G$ which is $\kappa$-rigid.

\sta{\label{st:notanti}There is a
$\left(10(\delta+3t\kappa)^{t+3}\lambda^3\right)^t$-subset $\mca{L}_0$ of $\mca{L}$ such that no two paths $L,L'\in \mca{L}_0$ are anticomplete in $G$.}

Let $\Gamma$ be a graph with $V(\Gamma)=\mca{L}$ such that for distinct $L,L'\in \mca{L}$, we have $LL'\in E(\Gamma)$ if and only if $L$ and $L'$ are not anticomplete in $G$. Since $|V(\Gamma)|=|\mca{L}|=f_{\ref{thm:seedling_branches}}$, it follows from Theorem~\ref{thm:classicalramsey} that $\Gamma$ contains either a stable set of cardinality $\kappa$ or a clique of cardinality $(10(\delta+3t\kappa)^{t+3}\lambda^3)^t$. In the former case, there is a set $\mca{K}\subseteq \mca{L}$ of $\kappa$ pairwise anticomplete $(N(A),Y)$-paths in $G$, a contradiction to the assumption that $(A,\mca{L},Y)$ is $\kappa$-rigid. So the latter case holds; that is, there is a $(10(\delta+3t\kappa)^{t+3}\lambda^3)^t$-subset $\mca{L}_0$ of $\mca{L}$ such that no two paths $L,L'\in \mca{L}_0$ are anticomplete in $G$. This proves \eqref{st:notanti}.
\medskip

Henceforth, let $\mca{L}_0\subseteq \mca{L}$ be as given by \eqref{st:notanti}. In particular, $\mca{L}_0$ is a set of pairwise disjoint $(N(A),Y)$-paths in $G$, and so each path $L\in \mca{L}_0$ has an $N(A)$-end $x_L$ and a $Y$-end $y_L$. We apply Lemma~\ref{lem:magic} to $\mca{L}_0$ along with the labeling $x_L,y_L$ of the ends of each path $L\in \mca{L}_0$. Since $G$ is $K_{t+1}$-free and $|\mca{L}_0|=(10(\delta+3t\kappa)^{t+3}\lambda^3)^t$, we deduce that:

\sta{\label{st:magicapp}
There is a $(\delta+3t\kappa)$-subset $\mca{P}$ of $\mca{L}_0$, a vertex $z_{P}\in P$ for each $P\in \mca{P}$, and $\delta+3t\kappa$ pairwise disjoint $\lambda$-subsets $(\mca{Q}_P:P\in \mca{P})$ of $\mca{L}_0\setminus \mca{P}$, such that the following hold.
    \begin{itemize}
    \item The paths $(x_{P}\dd P\dd z_{P}: P\in \mca{P})$ are pairwise anticomplete in $G$.
    \item For each $P\in \mca{P}$, every $Q\in \mca{Q}_P$ contains a vertex $w_{Q}$ distinct from $x_{Q}$ such that $w_{Q}$ is the only vertex in $w_{Q}\dd Q\dd y_{Q}$ with a neighbor in $x_{P}\dd P\dd z_{P}$.
    \end{itemize}}

From now on, let $\mca{P}$ and $(z_P,\mca{Q}_P:P\in \mca{P})$ be as given by \eqref{st:magicapp}. For every $P\in \mca{P}$, let
    $$A_P=x_P\dd P\dd z_P;$$
    $$\mca{L}_P=\{w_{Q}\dd Q\dd y_{Q}:Q\in \mca{Q}_P\};$$
    $$Y_P=\{y_Q:Q\in \mca{Q}_P\}.$$
Then $A_P\subseteq P$ is a path in $G\setminus A$ and $Y_P\subseteq V(G\setminus A)\setminus A_P$. Moreover, by the second bullet of \eqref{st:magicapp}, for every $Q\in \mca{Q}_P$, the path $w_Q\dd Q\dd y_Q$ is an $(N(A_{P}),Y_P)$-path in $G\setminus A_P$, and so $\mca{L}_P$ is a set of $\lambda$ pairwise disjoint $(N(A_{P}),Y_P)$-paths in $G\setminus A_P$. Note also that by construction, $(A_P\cup V(\mca{L}_P):P\in \mca{P})$ are pairwise disjoint subsets of
$V(\mca{L}_0)\subseteq V(G\setminus A)$ (note that $Y_P\subseteq V(\mca{L}_P)$ for all $P\in \mca{P}$). Therefore,

\sta{\label{st:trivialseedling} The triples $((A_P,\mca{L}_P,Y_P):P\in \mca{P})$ are pairwise disjoint $\lambda$-seedlings in
$G\setminus A$.}

We also show that:

\sta{\label{st:seedlingclose}The following hold.
\begin{itemize}
    \item The paths $(A_P:P\in \mca{P})$ are pairwise anticomplete in $G$.
    \item For every $P\in \mca{P}$, we have:
    \begin{itemize}
        \item $A$ and $A_P$ are not anticomplete in $G$; and
        \item $A$ and $V(\mca{L}_P)$ are anticomplete in $G$.
    \end{itemize}
\end{itemize}}

The first assertion is immediate from the first bullet of \eqref{st:magicapp}. We prove the second assertion. Let $P\in \mca{P}$. Then $x_P\in A_P$ has a neighbor in $A$ (because $x_P$ is the $N(A)$-end of $P$), and so $A$ and $A_P$ are not anticomplete in $G$. Moreover, by the second bullet of \eqref{st:magicapp}, we have
$$V(\mca{L}_P)\subseteq \bigcup_{Q\in \mca{Q}_P}(Q\setminus \{x_Q\}).$$
Recall also that each path $Q\in \mca{Q}_P$ is an $(N(A),Y)$-path in $G\setminus A$ where $x_Q$ is the $N(A)$-end of $Q$, and so $A$ and $Q\setminus \{x_Q\}$ are anticomplete in $G$. It follows that $A$ and $V(\mca{L}_P)$ are anticomplete in $G$. This proves \eqref{st:seedlingclose}.

\sta{\label{st:seedlingrigid}
There are $P_1,\ldots,P_{\delta}\in \mca{P}$ such that $(A_{P_i},\mca{L}_{P_i},Y_{P_i})$ is $g_{\ref{thm:seedling_branches}}$-rigid for every $i\in \poi_{\delta}$.}

Suppose not. Recall that by \eqref{st:magicapp}, we have $|\mca{P}|=\delta+3t\kappa$. So there is a $3t\kappa$-subset $\mca{A}$ of $\mca{P}$ such that for every $P\in \mca{A}$, the $\lambda$-seedling $(A_{P},\mca{L}_{P},Y_{P})$ in $G\setminus A$ is not $g_{\ref{thm:seedling_branches}}$-rigid. By definition, this means that for every $P\in \mca{A}$, there is a set $\mca{K}_P$ of $g_{\ref{thm:seedling_branches}}$ pairwise anticomplete $(N(A_P),Y_P)$-paths in $G$ with $V(\mca{K}_P)\subseteq V(\mca{L}_P)$. Also, by the first bullet of \eqref{st:seedlingclose}, the paths $(A_P:P\in \mca{A})$ are pairwise anticomplete. Since $G$ has no induced $K_{t,t}$-model, and since $g_{\ref{thm:seedling_branches}}=f_{\ref{lem:bigramsey}}(\kappa,1,t)$, it follows from Lemma~\ref{lem:bigramsey} that there are $P_1,\ldots,P_{\kappa}\in \mca{A}$ as well as $K_i\in \mca{K}_{P_i}$ for each $i\in \poi_{\kappa}$, such that $(A_{P_i}\cup K_i:i\in \poi_{\kappa})$ are pairwise anticomplete in $G$. Now, recall that for each $i\in \poi_{\kappa}$, the vertex $x_{P_i}\in A_{P_i}$ has a neighbor in $A$, and $K_i\in \mca{K}_{P_i}$ is an $(N(A_{P_i}),Y_{P_i})$-path in $G$, which in turn implies that $K_i$ has an end in $Y_{P_i}\subseteq Y$. Consequently, for every $i\in \poi_{\kappa}$, there is an $(N(A),Y)$-path $L_i$ in $G$ with $L_i\subseteq A_{P_i}\cup K_i$. Moreover, we have
$$A_{P_i}\cup K_i\subseteq P_i\cup V(\mca{K}_{P_i})\subseteq P_i\cup V(\mca{L}_{P_i})\subseteq P_i\cup V(\mca{Q}_{P_i})$$
and so by \eqref{st:magicapp}, we have $A_{P_i}\cup K_i\subseteq V(\mca{L}_0)\subseteq V(\mca{L})$. But now $\mca{K}=\{L_1,\ldots,L_{\kappa}\}$ is a set of $\kappa$ pairwise anticomplete $(N(A),Y)$-paths in $G$ with $V(\mca{K})\subseteq V(\mca{L})$, violating the assumption that the seedling $(A,\mca{L},Y)$ is $\kappa$-rigid. This proves \eqref{st:seedlingrigid}.
\medskip

Let $\{P_1,\ldots,P_{\delta}\}\subseteq \mca{P}$ be as given by \eqref{st:seedlingrigid}. For each $i\in \poi_{\delta}$, let $A_i=A_{P_i}$, let $\mca{L}_i=\mca{L}_{P_i}$, and let $Y_i=Y_{P_i}$. From \eqref{st:trivialseedling}, \eqref{st:seedlingclose}, and
\eqref{st:seedlingrigid}, it follows that $(A_1,\mca{L}_1,Y_1),\ldots,(A_{\delta},\mca{L}_{\delta},Y_{\delta})$ are $\delta$ pairwise disjoint $\lambda$-seedlings in $G\setminus A$ satisfying both \ref{thm:seedling_branches}\ref{thm:seedling_branches_a} and \ref{thm:seedling_branches}\ref{thm:seedling_branches_b}. This completes the proof of Theorem~\ref{thm:seedling_branches}.
\end{proof}

\section{From a seedling to a tree}\label{sec:seedlingtotree}

In this section, we prove Theorem~\ref{thm:seedling_to_tree}. As shown at the end of Section~\ref{sec:outline}, this will conclude the proof of Theorem~\ref{thm:main_tree_indm}.

\seedlingtotree*

\begin{proof}
First, for each $r\in \poi$, we define a function $$\xi_{r}:\poi^3\rightarrow \poi.$$ The definition is recursive in $r$, as follows. For $r=1$, let
$$\xi_{1}(a,b,c)=b^{a}$$ for every $(a,b,c)\in \poi^3$. For $r\geq 2$, assuming the function $\xi_{r-1}$ is defined, let
$$\xi_{r}(a,b,c)=f_{\ref{thm:seedling_branches}}(a,2ab,\xi_{r-1}(a,f_{\ref{lem:bigramsey}}(b,b,a),g_{\ref{thm:seedling_branches}}(a,c)),c)$$
for every $(a,b,c)\in \poi^3$. This concludes the definition of the functions $(\xi_r:r\in \poi)$.
\medskip

Back to the proof of \ref{thm:seedling_to_tree}, we will prove by induction on $r\in \poi$ that, for all $d,t,\kappa\in \poi$,

$$f_{\ref{thm:seedling_to_tree}}=f_{\ref{thm:seedling_to_tree}}(d,r,t,\kappa)=\xi_r(t,d,\kappa)$$
satisfies the theorem.

Let $G$ be a $t$-tidy graph and let $(A,\mca{L},Y)$ be an $f_{\ref{thm:seedling_to_tree}}$-seedling in $G$ which is $\kappa$-rigid. Assume that $r=1$. Then we have $|\mca{L}|=f_{\ref{thm:seedling_to_tree}}(d,1,t,\kappa)=\xi_1(t,d,\kappa)=d^t$. For each $L\in \mca{L}$, let $x_L$ be the $N(A)$-end of $L$. Since $G$ is $K_{t+1}$-free, it follows from Theorem~\ref{thm:classicalramsey} that there is a $d$-subset $\{L_1,\ldots,L_d\}$ of $\mca{L}$ for which $\{x_{L_1},\ldots,x_{L_d}\}$ is a stable set in $G$. But now $(A,\{x_{L_1}\},\ldots,\{x_{L_d}\})$ is an induced $T_{d,1}$-model in $G$ where $A$ is the branch set associated with the root of $T_{d,1}$.

Therefore, we may assume that $r\geq 2$. Let
$$\Delta=f_{\ref{lem:bigramsey}}(d,d,t).$$
In particular, we have $\Delta\geq d$. Let $$\Lambda=\xi_{r-1}(t,\Delta,g_{\ref{thm:seedling_branches}}(t,\kappa)).$$
Then we have $|\mca{L}|=f_{\ref{thm:seedling_to_tree}}(d,r,t,\kappa)=\xi_r(t,d,\kappa)=f_{\ref{thm:seedling_branches}}(t,2dt,\Lambda,\kappa)$. Applying Theorem~\ref{thm:seedling_branches} to $(A,\mca{L},Y)$, we deduce that:

\sta{\label{st:seedingbranchesapp} There are $2dt$ pairwise disjoint $\Lambda$-seedlings $(A_1,\mca{L}_1,Y_1),\ldots,(A_{2dt},\mca{L}_{2dt},Y_{2dt})$
in $G\setminus A$ with the following specifications.
    \begin{itemize}
    \item The paths $A_1,\ldots,A_{2dt}$ are pairwise anticomplete in $G$.
    \item For every $i\in \poi_{2dt}$, we have:
    \begin{itemize}
        \item $A$ and $A_i$ are not anticomplete in $G$;
        \item $A$ and $V(\mca{L}_i)$ are anticomplete in $G$; and
        \item $(A_i,\mca{L}_i,Y_i)$ is $g_{\ref{thm:seedling_branches}}(t,\kappa)$-rigid.
    \end{itemize}
    \end{itemize}}
Moreover, for every $i\in \poi_{2dt}$, since $(A_i,\mca{L}_i,Y_i)$ is a $g_{\ref{thm:seedling_branches}}(t,\kappa)$-rigid $\xi_{r-1}(t,\Delta,g_{\ref{thm:seedling_branches}}(t,\kappa))$-seedling in $G$, it follows from the inductive hypothesis applied to $(A_i,\mca{L}_i,Y_i)$ that there is an induced $T_{\Delta,r-1}$-model $(A_{i,v}:v\in V(T_{\Delta,r-1}))$ in $G[A_i\cup V(\mca{L}_i)]$ where $A_i$ is the branch set associated with the root of $T_{\Delta,r-1}$.

Let $u_0$ be the root of $T_{\Delta,r-1}$; thus, we have $A_i=A_{i,u_0}$ for every $i\in \poi_{2dt}$. Let $u_1,\ldots,u_{\Delta}$ be the neighbors of $u_0$ in $T_{\Delta,r-1}$ and let $T_1,\ldots,T_{\Delta}$ be the components of $T_{\Delta,r-1}\setminus \{u_0\}$ containing $u_1,\ldots,u_{\Delta}$, respectively. It follows that $T_i$, for each $i\in \poi_{\Delta}$, is isomorphic to $T_{\Delta,r-2}$ with root $u_i$ (in particular, $T_i$ is non-null because $r\geq 2$). Moreover, since $\Delta\geq d$, it follows that for every $i\in \poi_{\Delta}$, there is an induced subgraph $T^i$ of $T_i$ isomorphic to $T_{d,r-2}$ with root $u_i$.

For each $i\in \poi_{2dt}$ and every $j\in \poi_{\Delta}$, let
$$B_i^j=\bigcup_{v\in V(T^j)}A_{i,v}.$$

Then $B_i^1,\ldots,B_i^{\Delta}$ are pairwise anticomplete in $G$ for every $i\in \poi_{2dt}$. Also, by the first bullet of \eqref{st:seedingbranchesapp}, the sets $A_1,\ldots,A_{2dt}$ are pairwise anticomplete in $G$. Since $\Delta=f_{\ref{lem:bigramsey}}(d,d,t)$, it follows from Lemma~\ref{lem:bigramsey} that there is a $d$-subset $I$ of $\poi_{2dt}$ as well as a $d$-subset $J_i\subseteq \poi_{\Delta}$ for each $i\in I$, such that the $d$ sets
$$\left(A_i\cup \left(\bigcup_{j\in J_i}B^j_i\right):i\in I\right)$$
are pairwise anticomplete in $G$.

Now, for every $i\in I$, the subgraph of $T_{\Delta,r-1}$ induced by $\{u_0\}\cup(\bigcup_{j\in J_i}V(T^j))$ is isomorphic to $T_{d,r-1}$. Also,
by the second bullet of \eqref{st:seedingbranchesapp}, the sets $A$ and $A_i=A_{i,u_0}$ are not anticomplete in $G$, whereas $A$ and $\bigcup_{j\in J_i}\bigcup_{v\in V(T^j)}A_{i,v}=\bigcup_{j\in J_i}B^j_i\subseteq V(\mca{L}_i)$ are anticomplete in $G$. Hence,
$$\left(A; A_{i,v}:i\in I, v\in \{u_0\}\cup \left(\bigcup_{j\in J_i}V(T^j)\right)\right)$$
is an induced $T_{d,r}$-model in $G[A\cup V(\mca{L})]$ where $A$ is the branch set associated with the root of $T_{d,r}$. This completes the proof of Theorem~\ref{thm:seedling_to_tree}.
\end{proof}


\bibliographystyle{amsplain}


\begin{aicauthors}
\begin{authorinfo}[mchud]
 Maria Chudnovsky\\
 Princeton University\\
 Princeton, New Jersey, USA
\end{authorinfo}
\begin{authorinfo}[hajebi]
Sepehr Hajebi\\
Department of Combinatorics and Optimization, University of Waterloo\\
Waterloo, Ontario, Canada
\end{authorinfo}
\begin{authorinfo}[spirkl]
 Sophie Spirkl\\
Department of Combinatorics and Optimization, University of Waterloo\\
Waterloo, Ontario, Canada
\end{authorinfo}
\end{aicauthors}

\end{document}